\documentclass[12pt,oneside,english]{amsart}
\usepackage[T1]{fontenc}
\usepackage[utf8]{inputenc}
\usepackage{color}
\usepackage{babel}
\usepackage{amstext}
\usepackage{amsthm}
\usepackage{amssymb}
\usepackage{cancel}
\usepackage{verbatim}
\usepackage{geometry}
\usepackage[pdfusetitle,bookmarks=true,bookmarksnumbered=false,bookmarksopen=false,
 breaklinks=false,pdfborder={0 0 1},backref=false,colorlinks=true]
 {hyperref}
\hypersetup{
 pdfborderstyle=}

\makeatletter
\numberwithin{equation}{section}
\numberwithin{figure}{section}
\theoremstyle{plain}
\newtheorem{thm}{\protect\theoremname}
\theoremstyle{plain}
\newtheorem{prop}{\protect\propositionname}
\theoremstyle{plain}
\newtheorem{lem}[thm]{\protect\lemmaname}
\theoremstyle{remark}
\newtheorem{rem}[thm]{\protect\remarkname}
\theoremstyle{definition}
\newtheorem{defn}[thm]{\protect\definitionname}

\newcommand{\R}{\mathbb{R}}

\renewcommand{\d}{\displaystyle}

\providecommand{\lemmaname}{Lemma}
\providecommand{\remarkname}{Remark}
\providecommand{\theoremname}{Theorem}
\providecommand{\propositionname}{Proposition}

\allowdisplaybreaks

\makeatother

\providecommand{\definitionname}{Definition}
\providecommand{\lemmaname}{Lemma}
\providecommand{\remarkname}{Remark}
\providecommand{\theoremname}{Theorem}

\begin{document}
\title{A Christ-Fefferman type approach to the one sided maximal operator }
\author{Francisco J. Martín-Reyes}
\author{Israel P. Rivera-Ríos}
\address{Francisco J. Martín-Reyes, Israel P. Rivera-Ríos. Departamento de Análisis Matemático, Estadística
e Investigación Operativa y Matemática Aplicada. Facultad de Ciencias.
Universidad de Málaga (Málaga, Spain).}
\email{martin\_reyes@uma.es}
\email{israelpriverarios@uma.es}
\author{Pablo Rodríguez-Padilla}
\address{Pablo Rodríguez-Padilla. Departamento de Ciencias Integradas. Facultad de Ciencias Experimentales. Universidad de Huelva (Huelva, Spain)}
\email{rodriguezpadillapablo@uma.es}
\begin{abstract}
In this paper, an approach to the one sided maximal function in the
spirit of the Christ-Fefferman \cite{CF} proof for the strong type
weighted estimates of the maximal function is provided. As applications
of that approach, we provide an alternative proof of the sharp weighted
estimate for the one sided maximal function that was settled by one
of us and de la Torre \cite{MRdT}, a one sided two weight bumps counterpart
of a result of Pérez and Rela \cite{PR}, and also one sided counterparts
of some very recent mixed weak type results due to Sweeting \cite{Sw}.
\end{abstract}
\maketitle
\section{Introduction and Main results}
We recall that the Hardy-Littlewood maximal operator is defined as
\begin{equation}
Mf(x)=\sup_{x\in Q}\frac{1}{|Q|}\int_{Q}|f(y)|dy\label{eq:HL}
\end{equation}
and that for $1<p<\infty$, $w\in A_{p}$ if 
\[
[w]_{A_{p}}=\sup_{Q}\frac{1}{|Q|}\int_{Q}w\left(\frac{1}{|Q|}\int_{Q}w^{-\frac{1}{p-1}}\right)^{p-1}<\infty.
\]
In the equations above, $Q$ stands for cubes with their sides parallel to the coordinate axis.

Since the seminal Muckenhoupt paper \cite{M}, in which he showed
that the $A_{p}$ condition stated above characterizes, for $1<p<\infty$,
the weighted $L^{p}$ boundedness of the Hardy-Littlewood maximal
function, some further proofs of the sufficiency have been provided.
Among them, Christ and Fefferman \cite{CF}, showed that 
\begin{equation}
\|Mf\|_{L^{p}(w)}\leq c_{p,w}\|f\|_{L^{p}(w)}.\label{eq:Buck}
\end{equation}
It is worth noting that it is easy to track the dependence on the
$A_{p}$ constant in their argument and that applying it to the dyadic
maximal function, one has that actually $c_{p,w}\lesssim c_{p}[w]_{A_{p}}^{\frac{1}{p-1}}$,
which is sharp in  terms of the exponent of the $A_{p}$ constant.

Probably, the main highlight of the Christ-Fefferman argument is that
they managed to avoid the usage of the reverse Hölder property of $A_{p}$ weights in their proof. To do that, they relied upon an approach that
could be regarded as one of the first sparse domination results available
in the literature. Let us further expand on this. A key point in their argument is the fact that
\[
\|Mf\|_{L^{p}(x)}\leq\left(\sum_{j,k}\left(\frac{1}{|Q_{j}^{k}|}\int_{Q_{j}^{k}}|f(y)|dy\right)^{p}w(E_{Q_{j}^{k}})\right)^{\frac{1}{p}}
\]
where $Q_{j}^{k}$ are the Calderón-Zygmund cubes of $Mf$ at height
$C_{n}^{k}$ for every integer $k$, where $C_{n}>0$ is a constant
to be chosen, and 
\[
E_{Q_{j}^{k}}=Q_{j}^{k}\setminus\bigcup_{l>k,Q_{i}^{l}\subset Q_{j}^{k}}Q_{i}^{l}\subset Q_{j}^{k}.
\]
It is clear that the sets $E_{Q_{j}^{k}}$ are pairwise disjoint.
Choosing $C_{n}$ large enough, also $|Q_{j}^{k}|\leq2|E_{Q_{j}^{k}}|$.
Nowadays, we call families of cubes having these properties sparse
families.

Let us turn our attention now to the one sided setting. We recall
that the one sided maximal operators $M^{+}$ and $M^{-}$ are defined
as 
\[
M^{+}f(x)=\sup_{h>0}\frac{1}{h}\int_{x}^{x+h}|f(y)|dy\qquad M^{-}f(x)=\sup_{h>0}\frac{1}{h}\int_{x-h}^{x}|f(y)|dy
\]
and that for $1<p<\infty$, the $A_{p}^{+}$ and $A_{p}^{-}$ classes
characterize the weighted $L^{p}$ boundedness of $M^{+}$ and of
$M^{-}$ respectively (see \cite{S}) and are defined as 
\begin{align*}
[w]_{A_{p}^{+}} & =\sup_{a<b<c}\frac{1}{c-a}\int_{a}^{b}w\left(\frac{1}{c-a}\int_{b}^{c}w^{-\frac{1}{p-1}}\right)^{p-1},\\{}
[w]_{A_{p}^{-}} & =\sup_{a<b<c}\frac{1}{c-a}\int_{b}^{c}w\left(\frac{1}{c-a}\int_{a}^{b}w^{-\frac{1}{p-1}}\right)^{p-1}.
\end{align*}
It is worth noting that $A_{p}=A_{p}^{+}\cap A_{p}^{-}$, and that
$A_{p}\subsetneq A_{p}^{+}$ and $A_{p}\subsetneq A_{p}^{-}$. We
recall as well that the $A_{\infty}^{+}$ and $A_{\infty}^{-}$ are
defined as follows. 
\[
[w]_{A_{\infty}^{+}}=\sup_{a<b}\frac{1}{w(a,b)}\int_{a}^{b}M^{-}(w\chi_{(a,b)}),\qquad[w]_{A_{\infty}^{-}}=\sup_{a<b}\frac{1}{w(a,b)}\int_{a}^{b}M^{+}(w\chi_{(a,b)}).
\]

In the last years, quantitative weighted estimates 
have been a very active area of research, leading to important developments
in the theory, such as the sparse domination theory. An important role
in that trend was played by what is known now as the $A_{2}$ theorem,
formerly $A_{2}$ conjecture, that was settled by Hytönen \cite{H}. That result states that if $T$ is a Calderón-Zygmund operator then 
\begin{equation}
\|Tf\|_{L^{2}(w)}\leq c_{T}[w]_{A_{2}}\|f\|_{L^{2}(w)}.\label{eq:A2}
\end{equation}
In the one sided setting, some main questions regarding quantitative
estimates, such as the $A_{2}$ conjecture for one sided Calderón-Zygmund
operators, namely if (\ref{eq:A2}) holds for one sided Calderón-Zygmund
operators replacing $A_{2}$ by its corresponding one sided counterpart,
remain open. There are difficulties in the dyadic approach, that has
been very fruitful in the classical setting, that have not allowed
yet to transfer it to this setting. Trying to, somehow push the one
sided theory in that direction one reasonable first question is if
the Christ-Fefferman argument can be adapted in that way. In this
work we provide an argument that mimics, in some sense, the ideas
of the Christ-Fefferman approach even tough dyadic structures are not used.

As a first application of that approach we give a new proof of the
sharp bound for the maximal function that was obtained by de la Torre and the first author in \cite{MRdT}. 
\begin{thm}
\label{thm:Fuerte}Let $1<p<\infty$ and $w\in A_{p}^{+}$. Then 
\[
\|M^{+}f\|_{L^{p}(w)}\leq c_{p}[w]_{A_{p}^{+}}^{\frac{1}{p-1}}\|f\|_{L^{p}(w)}.
\]
Furthermore, 
\begin{equation}
\|M^{+}f\|_{L^{p}(w)}\leq c_{p}([w]_{A_{p}^{+}}[\sigma]_{A_{\infty}^{-}})^{\frac{1}{p}}\|f\|_{L^{p}(w)}\label{eq:mixedconstant}
\end{equation}
where $\sigma=w^{-\frac{1}{p-1}}$.
\end{thm}

At this point it is worth noting that we are going to derive \eqref{eq:mixedconstant}
from a more general two weight bumps estimate which is a one sided
counterpart of a result due to Rela and Pérez \cite{PR}. With respect to the one sided setting, our result can
also be regarded as a quantitative revisit to a work by Riveros, de
Rosa and de la Torre \cite{RdRdT}. We remit the reader to Section
\ref{sec:2W} for more details. 

Let us turn our attention now to mixed
weighted estimates. The study of that kind of estimates began in the
seminal paper by Muckenhoupt and Wheeden \cite{MW} in which they
dealt with inequalities of the form 
\begin{equation}
\left|\left\{ x\in\mathbb{R}:w^{\frac{1}{p}}(x)|Gf(x)|>t\right\} \right|\leq c\int_{\mathbb{R}}\frac{|f|^{p}}{t^{p}}w\label{eq:MWmixed}
\end{equation}
where $G$ stands either for the Hardy-Littlewood maximal operator
or for the Hilbert transform. Later on, Sawyer \cite{SMixed} studied
some related inequalities in the case $p=1$. Since Sawyer's result,
a number of works have been devoted to further understand that kind
of estimates. For further details we remit the reader to \cite{MR2172941,MR3498179,MR3557137,MR3961329,MR3987919,MR3990170,MR4002540,MR4140763,MR4166766,MR4421920,MR4533037,MR4614637,MR4694578,MR4732437}.

In the last years, due to the fact that until a very recent work of
Nieraeth \cite{NArxiv} mixed weak type estimates were the most suitable
way available to have weak type estimates in the matrix weighted setting for operators such as the maximal function, 
(see \cite{CUIMPRR,CUS}), there has been a renewed interest in (\ref{eq:MWmixed}). See for instance \cite{LLORR,LLORR2}. Very recently Sweeting \cite{Sw} showed that if $1<p<\infty$,
such an estimate holds for $G$ being the Hardy Littlewood maximal
function if and only if 
\[
[w]_{A_{p}^{*}}=\sup_{Q}\frac{1}{|Q|^{p}}\|w\chi_{Q}\|_{L^{1,\infty}}\left(\int_{Q}\sigma\right)^{p-1}<\infty.
\]
That $[w]_{A_{p}^{*}}<\infty$ is a necessary condition had already
been established in the aforementioned work by Muckehoupt and Wheeden
\cite{MW}. Sweeting settled the sufficiency of $A_p ^*$ and provided a counterpart for fractional maximal functions as well.

Let us turn our attention now to our contribution. In this work we
provide a one sided counterpart of Sweeting's results. Let us begin
with a definition first.

Given $1<p<\infty$ we say that $w\in A_{p}^{+,*}$ if 
\[
[w]_{A_{p}^{+,*}}=\sup_{a<b<c}\frac{1}{(c-a)^{p}}\|w\chi_{(a,b)}\|_{L^{1,\infty}}\left(\int_{b}^{c}w^{-\frac{1}{p-1}}\right)^{p-1}<\infty
\]

Our result in the case of the one sided maximal operator is the following. 
\begin{thm}
\label{thm:mixed}Let $1<p<\infty$. We have that $w\in A_{p}^{+,*}$
if and only if 
\[
\|w^{\frac{1}{p}}M^{+}f\|_{L^{p,\infty}}\leq c_{w}\|fw^{\frac{1}{p}}\|_{L^{p}}
\]
Furthermore, $c[w]_{A_{p}^{+,*}}^{\frac{1}{p}}\leq c_{w}\leq c'[w]_{A_{p}^{+,*}}^{\frac{2}{p}}$
for some $c,c'>0$ independent of $w$. 
\end{thm}

Recall that, in the case of the fractional maximal operator, the one sided
versions are defined, for $0\leq\alpha<1$, as 
\[
M_{\alpha}^{+}f(x)=\sup_{h>0}\frac{1}{h^{1-\alpha}}\int_{x}^{x+h}|f(y)|dy\qquad M_{\alpha}^{-}f(x)=\sup_{h>0}\frac{1}{h^{1-\alpha}}\int_{x-h}^{x}|f(y)|dy
\]
If we let $1<p,q<\infty$ then we say that $w\in A_{p,q}^{+,*}$ if
\[
[w]_{A_{p,q}^{+,*}}=\sup_{a<b<c}\left(\frac{1}{c-a}\|w^{q}\chi_{(a,b)}\|_{L^{1,\infty}}\right)^{\frac{1}{q}}\left(\frac{1}{c-a}\int_{b}^{c}w^{-p'}\right)^{\frac{1}{p}}<\infty
\]

In this case we have the following result. 
\begin{thm}
\label{thm:mixedFrac}Let $0<\alpha<1$, $1<p<\frac{1}{\alpha}$ ,
$\frac{1}{q}=\frac{1}{p}-\alpha$. We have that $w\in A_{p,q}^{+,*}$
if and only if 
\[
\|wM_{\alpha}^{+}f\|_{L^{q,\infty}}\leq c_{w}\|fw\|_{L^{p}}.
\]
Furthermore, $c[w]_{A_{p,q}^{+,*}}\leq c_{w}\leq c'[w]_{A_{p,q}^{+,*}}^{2}$,
for some $c,c'>0$ independent of $w$. 
\end{thm}

At this point it is worth noting that in contrast with Sweeting's
approach, we show that $w\in A_{p,q}^{+,*}$ implies the claimed estimate
reducing the problem to the case of the maximal operator via an inequality
that we borrow from \cite{MR3987919}.

At this point it seems convenient to gather some notation that has already appeared and some to appear yet. Given $1\leq p<\infty$ we define
\[
\begin{split}\|f\|_{L^{p,\infty}}&=\sup_{t>0}t\left|\left\{x\in\mathbb{R} : |f(x)|>t\right\}\right|^\frac{1}{p}\\
\|f\|_{L^{p,1}}&=\int_0^\infty\left|\left\{x\in\mathbb{R} : |f(x)|>t\right\}\right|^\frac{1}{p}dt\\
\end{split}\]
Related to this, it is straightforward to see that $\|f\|_{L^{p,\infty}}=\|f^p\|_{L^{1,\infty}}^{\frac{1}{p}}$.
We shall denote $A\lesssim B$ when there exists a constant $C>0$ that does not depend on the main parameters involved such that $A\leq C B$. Abusing notation if $w$ is a weight we denote $w(E)=\int_Ew(x)dx$.

The remainder of the paper is organized as follows. In Section \ref{sec:Lemmatta}
we provide some lemmatta required for the proofs of the main results.
In Section \ref{sec:Proofs-of-the} we give the proofs of the main
results. Finally in Section \ref{sec:2W} we provide the bump conditions
result that allows to derive \eqref{eq:mixedconstant}.

\section{Lemmatta}\label{sec:Lemmatta}

\subsection{A key \textquotedblleft sparse alike\textquotedblright{} lemma}

Since one sided $A_{p}$ type classes are larger than their classical
counterparts, one needs to ``extract'' more information related
to ``sparseness'' in order to be able to provide results for all
the weights in the class. An application of the the next lemma to
suitable intervals will allow us to keep that ``additional'' at
the cost of having uniformly bounded overlapping instead of being
pairwise disjoint in contrast with Christ-Fefferman $E_{j,k}$ sets. 
\begin{lem}
\label{lem:Key}Let $\lambda_{2}>\lambda_{1}>0$ and let us call 
\[
F=\left\{ z\in \mathbb{R}\,:\,M^{+}f(z)\leq\lambda_{2}\right\} .
\]
Assume that for  some $x\in\mathbb{R}$ we have that $M^{+}f(x)\leq\lambda_{1}$. Then
\[
|F\cap(x,y)|\geq\left(1-\frac{\lambda_{1}}{\lambda_{2}}\right)|(x,y)|
\]
for every $x\leq y$. 
\end{lem}

\begin{proof}
We begin considering 
\[
\left\{ z\in\mathbb{R}\,:\,M^{+}f(z)>\lambda_{2}\right\} =\bigcup_{i}I_{i}
\]
We recall that if $I_{i}=(a_{i},b_{i})$ then 
\[
\frac{1}{|I_i|}\int_{I_i}|f|=\lambda_{2}.
\]
and, obviously, $x\not\in I_{i}$. Now we let 
\[
H=\bigcup_{I_{i}\cap(x,y)\not=\emptyset}I_{i}.
\]
If $H=\emptyset$, the desired conclusion would hold trivially, since
$F\cap(x,y)=(x,y)$. Hence we may assume that $H\not=\emptyset$.
Now, we observe that if $I_{i}\cap(x,y)\not=\emptyset$ then $I_{i}\subset(x,\infty)$
since $x\not\in I_{i}$. \\
 Bearing this in mind, there are two possible cases. 
\begin{enumerate}
\item $H\subset(x,y).$ In this case we observe that by the properties of
the intervals $I_{i}$ and since $M^{+}f(x)\leq\lambda_{1}$ we have
that 
\begin{align*}
|H|=|H\cap(x,y)| & =\sum_{i}|I_{i}|=\frac{1}{\lambda_{2}}\sum_{i}\int_{I_{i}}|f|\leq\frac{1}{\lambda_{2}}\int_{x}^{y}|f|\\
 & =\frac{y-x}{\lambda_{2}}\frac{1}{y-x}\int_{x}^{y}|f|\leq(y-x)\frac{\lambda_{1}}{\lambda_{2}}.
\end{align*}
Taking this into account, 
\[
|(x,y)|=|F\cap(x,y)|+|H\cap(x,y)|\leq|F\cap(x,y)|+|(x,y)|\frac{\lambda_{1}}{\lambda_{2}}
\]
and hence 
\[
\left(1-\frac{\lambda_{1}}{\lambda_{2}}\right)|(x,y)|\leq|F\cap(x,y)|.
\]
\item $H\not\subset(x,y)$. Let us call $(c,d)$ the rightmost interval
$I_{i}$ that intersects. Since $H\not\subset(x,y)$ then $d>y>c$.
This yields that $(y,d)\subset H$ and consequently $F\cap(x,y)=F\cap(x,d)$.
The interval $(x,d)$ is in the situation of the first case. Hence
the same argument as above yields that 
\[
|H|\leq(d-x)\frac{\lambda_{1}}{\lambda_{2}}
\]
and, consequently, 
\[
|F\cap(x,y)|=|F\cap(x,d)|\geq\left(1-\frac{\lambda_{1}}{\lambda_{2}}\right)|(x,d)|\geq\left(1-\frac{\lambda_{1}}{\lambda_{2}}\right)|(x,y)|.
\]
\end{enumerate}
\end{proof}

\subsection{Lemmatta concerning weights}
The application of our first lemma will say that restricted $A_p$ on a weight implies that ``sparseness'' in Lebesgue measure is transmited to ``sparseness'' in terms of the weight itself.
\begin{lem}
\label{lem:RWtype}Let $1<r<\infty$. Let $\sigma\in A_{r}^{\mathcal{R},-}$,
namely, assume that
\[
[\sigma]_{A_{r}^{\mathcal{R},-}}=\sup\frac{|E|}{|(a,c)|}\left(\frac{\sigma(b,c)}{\sigma(E)}\right)^{\frac{1}{r}}<\infty.
\]
where the $\sup$ is taken over every $a<b<c$ and every measurable
set $E\subset(a,b)$. Assume that there exists $a_{0}\in\mathbb{R}$,
a set $A$ and $\eta\in(0,1)$ such that for every $z>a_{0}$, 
\[
|A\cap(a_{0},z)|>\eta|(a_{0},z)|.
\]
Then there exists $C>0$ independent of $A$ and $\sigma$ such that
for every $z>a_{0}$, 
\[
\sigma(a_{0},z)\leq C\left(\frac{[\sigma]_{A_{r}^{\mathcal{R},-}}}{\eta}\right)^{r}\sigma\left(A\cap(a_{0},z)\right).
\]
\end{lem}
\begin{proof}
By inspection of the proof of \cite[Lemma 3]{O} it follows that if there exists a constant $\kappa>0$ such that for every $a<b<c$ and every measurable subset $E\subset (a,b)$,
\begin{equation}\label{eq:ApREQuiv}
\frac{|E|}{|(a,c)|}\leq\kappa\left(\frac{\sigma(E)}{\sigma(b,c)}\right)^{\frac{1}{r}}
\end{equation}
then for every measurable set $J$   
\[
\sup_{t>0}t\sigma\left(\left\{ x\in\mathbb{R}\,:\,M^{-}(\chi_{J})>t\right\} \right)^{\frac{1}{r}}\leq C\kappa\|\chi_{J}\|_{L^{r}(\sigma)}.
\]
where $C>0$ is a constant independent of $J$ and $\sigma$.

Note that the least $\kappa>0$ that satisfies is precisely \eqref{eq:ApREQuiv}. Hence, the restricted weak type inequality we have stated actually holds replacing $\kappa$ by $[\sigma]_{A_{r}^{\mathcal{R},-}}$.

Let $z>a_{0}$. Note that if $x\in(a_{0},z)$, 
\begin{align*}
M^{-}(\chi_{A\cap(a_{0},z)})(x) & =\sup_{h>0}\frac{1}{h}\int_{x-h}^{x}\chi_{A\cap(a_{0},z)}=\sup_{h>0}\frac{1}{h}|A\cap(a_{0},z)\cap(x-h,x)|\\
 & \geq\frac{1}{x-a_{0}}|A\cap(a_{0},z)\cap(x-(x-a_{0}),x)|=\frac{1}{x-a_{0}}|A\cap(a_{0},x)|>\eta 
\end{align*}
Then, 
\begin{align*}
 & \sigma(a_{0},z)=\sigma\left(\left\{ y\in(a_{0},z)\ :\ M^{-}(\chi_{A\cap(a_{0},z)})(y)>\eta\right\} \right)\\
 & \leq C\left(\frac{[\sigma]_{A_{r}^{\mathcal{R},-}}}{\eta}\right)^{r}\int_{\mathbb{R}}\chi_{A\cap(a_{0},z)}^{r}\sigma=C\left(\frac{[\sigma]_{A_{r}^{\mathcal{R},-}}}{\eta}\right)^{r}\sigma\left(A\cap(a_{0},z)\right)
\end{align*}
and we are done. 
\end{proof}
In the following lemmatta, that we state and settle separatedly for reader's convenience, we prove that both $w\in A_{p}^+$ and $w\in A_{p}^{+,*}$ imply that $\sigma\in A_{r}^{\mathcal{R},-}$ for a suitable $r>1$.
Let us begin with the first of them, which combined with Lemma \ref{lem:Key} will allow us to settle Theorem \ref{thm:Fuerte}. 
\begin{lem}
\label{lem:CondWeightAp}Let $1<p<\infty$ and $\sigma\in A_{p'}^{-}$
. Then $[\sigma]_{A_{p'}^{\mathcal{R},-}}\leq[\sigma]_{A_{p'}^-}^\frac{1}{p'}$. Consequently, if there exists a measurable set $A$, $\eta\in(0,1)$
and $a\in\mathbb{R}$ such that for every $z>a$, 
\[
|A\cap(a,z)|>\eta|(a,z)|,
\]
then for every $z>a$ 
\[
\sigma(a,z)\lesssim\frac{1}{\eta^{p'}}[\sigma]_{A_{p'}^{-}}\sigma\left(A\cap(a,z)\right).
\]
\end{lem}

\begin{proof} The proof of the first inequality stated in this result is probably contained elsewhere, however we provide the argument for reader's convenience. As usual let us call $\sigma=w^{-\frac{1}{p-1}}$.
Let $a<b<c$ and let $E\subset(a,b)$ be a measurable set. Then we
have that 
\begin{align*}
|E| & \leq w(E)^{\frac{1}{p}}\sigma(E)^{\frac{1}{p'}}\\
 & \leq\frac{1}{|(a,c)|}w(a,b)^{\frac{1}{p}}\sigma(b,c)^{\frac{1}{p'}}\left(\frac{\sigma(E)}{\sigma(b,c)}\right)^{\frac{1}{p'}}|(a,c)|\\
 & \leq[w]_{A_{p}^{+}}^{\frac{1}{p}}\left(\frac{\sigma(E)}{\sigma(b,c)}\right)^{\frac{1}{p'}}|(a,c)|\\
 & =[\sigma]_{A_{p'}^{-}}^{\frac{1}{p'}}\left(\frac{\sigma(E)}{\sigma(b,c)}\right)^{\frac{1}{p'}}|(a,c)|
\end{align*}
and by Lemma \ref{lem:RWtype} we are done. 
\end{proof}
Our next lemma will be used in combination with Lemma \ref{lem:Key}
to settle Theorem \ref{thm:mixed}. 
\begin{lem}
\label{lem:CondWeightAp*}Let $1<p<\infty$ and $w\in A_{p}^{+,*}$.
Then the following statements hold.
\begin{enumerate}
    \item  If $s>1$, $w^{\frac{1}{s}}\in A_{p}^{+}$. Furthermore, $[w^\frac{1}{s}]_{A_p^+}\leq 2^p s'[w]_{A_{p}^{+,*}}^{\frac{1}{s}}$.
    \item  $[\sigma]_{A_{2p'}^{\mathcal{R},-}}\le8[w]_{A_{p}^{+,*}}^{\frac{1}{2p'}}$. Consequently, if there exists
$a\in\mathbb{R}$, a  measurable set $A$, and $\eta\in(0,1)$ such that
for every $z>a$, 
\[
|A\cap(a,z)|>\eta|(a,z)|,
\]
Then for every $z>a$, 
\[
\sigma(a,z)\lesssim\frac{1}{\eta^{2p'}}[w]_{A_{p}^{+,*}}\sigma\left(A\cap(a,z)\right).
\]
\end{enumerate}
\end{lem}

\begin{proof}Let us begin with the first part. Let $a<c$. Note that if $b=\frac{1}{2}(a+c)$, 
\[
\int_{a}^{b}w^{\frac{1}{s}}\leq\|w^{\frac{1}{s}}\chi_{(a,b)}\|_{L^{s,\infty}}\|\chi_{(a,b)}\|_{L^{s',1}}=s'\|w^{\frac{1}{s}}\chi_{(a,b)}\|_{L^{s,\infty}}(b-a)^{1-\frac{1}{s}}.
\]
Hence 
\[
\frac{1}{b-a}\int_{a}^{b}w^{\frac{1}{s}}\leq s'\left(\frac{1}{b-a}\|w\chi_{(a,b)}\|_{L^{1,\infty}}\right)^{\frac{1}{s}}.
\]
Consequently 
\begin{align*}
 & \left(\frac{1}{b-a}\int_{a}^{b}w^{\frac{1}{s}}\right)\left(\frac{1}{b-a}\int_{b}^{c}\sigma^{\frac{1}{s}}\right)^{p-1}\\
 & \leq s'\left(\frac{1}{b-a}\|w\chi_{(a,b)}\|_{L^{1,\infty}}\right)^{\frac{1}{s}}\left(\frac{1}{b-a}\int_{b}^{c}\sigma\right)^{\frac{p-1}{s}}\\
 & \leq s'[w]_{A_{p}^{+,*}}^{\frac{1}{s}}.
\end{align*}
An argument analogous to the one provided to settle Proposition \ref{prop:middleAp*} shows that if $\rho$ is a weight, then
\[[\rho]_{A_p^+}\leq 2^p
 \sup_{a<c,\ b=\frac{a+c}{2}} \left(\frac{1}{b-a}\int_{a}^{b}\rho\right)\left(\frac{1}{b-a}\int_{b}^{c}\rho^{-\frac{1}{p-1}}\right)^{p-1}\]
Taking this into account, we have shown that
\[[w^\frac{1}{s}]_{A_p^+}\leq 2^p s'[w]_{A_{p}^{+,*}}^{\frac{1}{s}}.\]
Let us focus now on the second part. Let $s>1$. If $E\subset(a,b)$ we have that 
\begin{align*}
\frac{|E|}{|(a,c)|} & =2\frac{1}{|(a,b)|}\int_{a}^{b}w^{\frac{1}{ps}}w^{-\frac{1}{ps}}\chi_{E}\\
 & \leq2\left(\frac{1}{b-a}\int_{a}^{b}w^{\frac{1}{s}}\right)^{\frac{1}{p}}\left(\frac{1}{b-a}\int_{a}^{b}\sigma^{\frac{1}{s}}\chi_{E}\right)^{\frac{1}{p'}}\\
 & \leq2\left(s'\right)^{\frac{1}{p}}\left(\frac{1}{b-a}\|w\chi_{(a,b)}\|_{L^{1,\infty}}\right)^{\frac{1}{sp}}\left(\frac{1}{b-a}\int_{a}^{b}\sigma\chi_{E}\right)^{\frac{1}{sp'}}\\
 & \leq2s'[w]_{A_{p}^{+,*}}^{\frac{1}{sp'}}\left(\frac{1}{b-a}\int_{b}^{c}\sigma\right)^{-\frac{1}{sp'}}\left(\frac{1}{b-a}\int_{a}^{b}\sigma\chi_{E}\right)^{\frac{1}{sp'}}\\
 & =2s'[w]_{A_{p}^{+,*}}^{\frac{1}{sp'}}\left(\frac{\sigma(E)}{\sigma(b,c)}\right)^{\frac{1}{sp'}}
\end{align*}
Hence, choosing $s=2$, 
\[
\frac{|E|}{|(a,c)|}\leq 4 [w]_{A_{p}^{+,*}}^{\frac{1}{2p'}}\left(\frac{\sigma(E)}{\sigma(b,c)}\right)^{\frac{1}{2p'}}.
\]
Now, if $z\in(a,c)$ and $E\subset(a,z)$, we have two cases.

\textbf{Case 1.} If $z\in(a,b)$ we let $\overline{a}=a-2(b-z)$. Note that
then $E\subset(\overline{a},z)$ and $z$ is the middle point of $(\overline{a},c)$.
Consequently 
\[
\frac{|E|}{|(\overline{a},c)|}\leq4[w]_{A_{p}^{+,*}}^{\frac{1}{2p'}}\left(\frac{\sigma(E)}{\sigma(z,c)}\right)^{\frac{1}{2p'}}
\]
and since $|(\overline{a},c)|\leq2|(a,c)|$, we have that 
\[
\frac{|E|}{|(a,c)|}\leq8[w]_{A_{p}^{+,*}}^{\frac{1}{2p'}}\left(\frac{\sigma(E)}{\sigma(z,c)}\right)^{\frac{1}{2p'}}.
\]
\textbf{Case 2.} If $z\in(b,c)$ we let $\overline{c}=c+2(z-b)$, then again
$E\subset(a,z)$ and $z$ is the middle point of $(a,\overline{c})$.
This yields
\[
\frac{|E|}{|(a,\overline{c})|}\leq4[w]_{A_{p}^{+,*}}^{\frac{1}{2p'}}\left(\frac{\sigma(E)}{\sigma(z,\overline{c})}\right)^{\frac{1}{2p'}}
\]
Note that since $|(a,\overline{c})|\leq2|(a,c)|$ and $\sigma(z,c)\leq\sigma(z,\overline{c})$
we have that 
\[
\frac{|E|}{|(a,c)|}\leq8[w]_{A_{p}^{+,*}}^{\frac{1}{2p'}}\left(\frac{\sigma(E)}{\sigma(z,c)}\right)^{\frac{1}{2p'}}.
\]
The arguments above imply that $[\sigma]_{A_{2p'}^{\mathcal{R},-}}\le8[w]_{A_{p}^{+,*}}^{\frac{1}{2p'}}$ and hence a direct application of Lemma \ref{lem:RWtype} leads to the desired
result. 
\end{proof}
We borrow our next result from \cite[Lemma 3]{MR3987919}, as we announced above. Our proof
is identical to theirs, however, we include it for reader's convenience. 
\begin{lem}
\label{lem:LemRed}Let $0<\alpha<1$, $1\leq p<\frac{1}{\alpha}$
, $\frac{1}{q}=\frac{1}{p}-\alpha$ and $s=1+\frac{q}{p'}$. Then
\[
M_{\alpha}^{+}(f)(x)\leq M^{+}(f^{p/s}w^{p/s-q/s})(x)^{s/q}\left(\int_{\mathbb{R}}|f(y)|^{p}w(y)^{p}dy\right)^{\alpha}.
\]
\end{lem}

\begin{proof}
To simplify the exposition, we shall assume that $f\geq0$. We set
$g=f^{\frac{p}{s}}w^{\frac{p}{s}-\frac{q}{s}}$. Note that then $f=g^{\frac{s}{p}}w^{\frac{q}{p}-1}$.
Let, $x\in\mathbb{R}$ and $h>0$. Then, by applying Hölder inequality
with exponents $\frac{1}{1-\alpha}$ and $\frac{1}{\alpha}$ 
\begin{align*}
\frac{1}{h^{1-\alpha}}\int_{x}^{x+h}f & =\frac{1}{h^{1-\alpha}}\int_{x}^{x+h}g^{\frac{s}{p}}w^{\frac{q}{p}-1}\\
 & =\frac{1}{h^{1-\alpha}}\int_{x}^{x+h}g^{1-\alpha}g^{\frac{s}{p}+\alpha-1}w^{q\alpha}\\
 & \leq\left(\frac{1}{h}\int_{x}^{x+h}g\right)^{1-\alpha}\left(\int_{x}^{x+h}g^{\frac{\frac{s}{p}+\alpha-1}{\alpha}}w^{q}\right)^{\alpha}
\end{align*}
Now, since $\frac{\frac{s}{p}+\alpha-1}{\alpha}=s$, we have that
\begin{align*}
\left(\int_{x}^{x+h}g^{\frac{\frac{s}{p}+\alpha-1}{\alpha}}w^{q}\right)^{\alpha} & =\left(\int_{x}^{x+h}g^{s}w^{q}\right)^{\alpha}=\left(\int_{x}^{x+h}f^{p}w^{p-q}w^{q}\right)^{\alpha}\\
 & \leq\left(\int_{\mathbb{R}}f^{p}w^{p}\right)^{\alpha}
\end{align*}
on the other hand, since $1-\alpha=\frac{s}{q}$ 
\[
\left(\frac{1}{h}\int_{x}^{x+h}g\right)^{1-\alpha}=\left(\frac{1}{h}\int_{x}^{x+h}f^{\frac{p}{s}}w^{\frac{p}{s}-\frac{q}{s}}\right)^{\frac{s}{q}}.
\]
Hence 
\[
\frac{1}{h^{1-\alpha}}\int_{x}^{x+h}f\leq\left(\frac{1}{h}\int_{x}^{x+h}f^{\frac{p}{s}}w^{\frac{p}{s}-\frac{q}{s}}\right)^{\frac{s}{q}}\left(\int_{\mathbb{R}}f^{p}w^{p}\right)^{\alpha}
\]
and taking $\sup$ in $h>0$ we are done. 
\end{proof}
We end up this section showing the relation between the $A_{p,q}^{+,*}$
and the $A_{s}^{+,*}$ class for a suitable $s>1$. 
\begin{lem}
\label{lem:ApqAs}Let $0<\alpha<1$, $1\leq p<\frac{1}{\alpha}$ ,
$\frac{1}{q}=\frac{1}{p}-\alpha$ and $s=1+\frac{q}{p'}$. Then $w\in A_{p,q}^{+,*}$
if and only if $w^{q}\in A_{s}^{+,*}.$ Furthermore, 
\[
[w]_{A_{p,q}^{+,*}}=[w^{q}]_{A_{s}^{+,*}}^{\frac{1}{q}}.
\]
\end{lem}

\begin{proof}
Recall that 
\[
[w]_{A_{p,q}^{+,*}}=\sup_{a<b<c}\left(\frac{1}{c-a}\|w^{q}\chi_{(a,b)}\|_{L^{1,\infty}}\right)^{\frac{1}{q}}\left(\frac{1}{c-a}\int_{b}^{c}w^{-p'}\right)^{\frac{1}{p'}}
\]
and that 
\[
[w^{q}]_{A_{s}^{+,*}}^{\frac{1}{q}}=\sup_{a<b<c}\left(\frac{1}{c-a}\right)^{s}\|w^{q}\chi_{(a,b)}\|_{L^{1,\infty}}\left(\int_{b}^{c}w^{q\frac{1}{1-s}}\right)^{s-1}.
\]
Bearing this in mind we argue as follows, 
\begin{align*}
 & \left(\frac{1}{c-a}\|w^{q}\chi_{(a,b)}\|_{L^{1,\infty}}\right)^{\frac{1}{q}}\left(\frac{1}{c-a}\int_{b}^{c}w^{-p'}\right)^{\frac{1}{p'}}\\
= & \left[\left(\frac{1}{c-a}\|w^{q}\chi_{(a,b)}\|_{L^{1,\infty}}\right)\left(\frac{1}{c-a}\int_{b}^{c}w^{q\frac{-p'}{q}}\right)^{\frac{q}{p'}}\right]^{\frac{1}{q}}\\
= & \left[\left(\frac{1}{c-a}\right)^{s}\|w^{q}\chi_{(a,b)}\|_{L^{1,\infty}}\left(\int_{b}^{c}w^{q\frac{1}{1-s}}\right)^{s-1}\right]^{\frac{1}{q}}.
\end{align*}
Consequently 
\[
[w]_{A_{p,q}^{+,*}}=[w^{q}]_{A_{s}^{+,*}}^{\frac{1}{q}}.
\]
\end{proof}

\section{Proofs of the main theorems}\label{sec:Proofs-of-the}

\subsection{Theorem \ref{thm:Fuerte} and sufficiency of Theorems \ref{thm:mixed}
and \ref{thm:mixedFrac}}

For all three proofs, we begin letting 
\[
O_{k}=\left\{ x\in\mathbb{R}\,:\,M^{+}f(x)>2^{k}\right\} .
\]
Again, for convenience, we assume that $f\geq0$. Since $O_{k}$ is
an open set, there exists a sequence $\{I_{j,k}\}_{j}$ of pairwise
disjoint intervals such that $O_{k}=\bigcup_{j}I_{j,k}$ and such
that 
\[
\frac{1}{b_{jk}-x}\int_{x}^{b_{jk}}f>2^{k}\qquad\text{for every }x\in I_{j,k}=(a_{jk},b_{jk}).
\]
We define 
\[
E_{j,k}=I_{j,k}\cap\left\{ x\,:\,M^{+}f(x)\leq2^{k+1}\right\} 
\]
and
\[
F_{k}=\left\{ x\in\mathbb{R}\,:\,M^{+}f(x)\leq2^{k+2}\right\} .
\]
Armed with the definitions above we begin providing our proof of Theorem
\ref{thm:Fuerte}

\subsubsection{Proof of Theorem \ref{thm:Fuerte}}
Note that the sets $E_{j,k}$ are pairwise disjoint
for every $j$ and every $k$. Then we have that 
\[
M^{+}f(x)  =\sum_{j,k}M^{+}f(x)\chi_{E_{j,k}}(x)\leq\sum_{j,k}2^{k+1}\chi_{E_{j,k}}(x).\]
Hence 
\[
\int_{\mathbb{R}}(M^{+}f(x))^{p}w(x)dx\leq2^{p}\sum_{j,k}2^{kp}w(E_{j,k}).
\]
Now we focus on each term $2^{kp}w(E_{j,k})$. Let us call $I_{j,k}=(a,b)$.
We split this interval as follows 
\[
\int_{x_{i+1}}^{b}\sigma=\int_{x_{i}}^{x_{i+1}}\sigma.
\]
where $x_{0}=a$. Consequently $\int_{x_{i}}^{b}\sigma=\frac{1}{2^{i}}\int_{a}^{b}\sigma.$
Then 
\begin{align*}
2^{kp}w(E_{j,k}) & =2^{kp}\sum_{i=0}^{\infty}w(E_{j,k}\cap(x_{i},x_{i+1}))
\end{align*}
If we call $\tilde{x_{i}}=\inf\left\{ z\in E_{j,k}\cap(x_{i},x_{i+1})\right\} $
and we take into account the properties of each $I_{j,k}$, 
\begin{align*}
 & 2^{kp}\sum_{i=0}^{\infty}w(E_{j,k}\cap(x_{i},x_{i+1}))=2^{kp}\sum_{i=0}^{\infty}w(E_{j,k}\cap(\tilde{x_{i}},x_{i+1}))\\
 & \leq\sum_{i=0}^{\infty}\left(\frac{1}{\int_{\tilde{x_{i}}}^{b}\sigma}\int_{\tilde{x_{i}}}^{b}f\right)^{p}\sigma(\tilde{x_{i}},b)^{p-1}\frac{1}{(b-\tilde{x_{i}})^{p}}w(E_{j,k}\cap(\tilde{x_{i}},x_{i+1}))\sigma(\tilde{x_{i}},b)\\
 & \leq2^{p-1}\sum_{i=0}^{\infty}\left(\frac{1}{\int_{\tilde{x_{i}}}^{b}\sigma}\int_{\tilde{x_{i}}}^{b}f\right)^{p}\sigma(x_{i+1},b)^{p-1}\frac{1}{(b-\tilde{x_{i}})^{p}}w(\tilde{x_{i}},x_{i+1})\sigma(\tilde{x_{i}},b)\\
 & \leq2^{p-1}[w]_{A_{p}^{+}}\sum_{i=0}^{\infty}\left(\frac{1}{\int_{\tilde{x_{i}}}^{b}\sigma}\int_{\tilde{x_{i}}}^{b}f\right)^{p}\sigma(\tilde{x_{i}},b)=(*)
\end{align*}
At this point note that since $\tilde{x}_{i}\in E_{j,k}$ then $Mf(\tilde{x}_{i})\leq2^{k+1}$
and then, by Lemma \ref{lem:Key} we have that
\[
|(\tilde{x}_{i},y)|\leq\frac{1}{1-\frac{2^{k+1}}{2^{k+2}}}|F_{k}\cap(\tilde{x}_{i},y)|=2|F_{k}\cap(\tilde{x}_{i},y)|
\]
for every $y>\tilde{x}_{i}$. This fact combined with Lemma \ref{lem:CondWeightAp}
yields\textcolor{black}{{} 
\[
\begin{aligned}\sigma(\tilde{x_{i}},b) & \leq\sigma(x_{i},b)=4\sigma(x_{i+1},x_{i+2})\\
 & \leq4\sigma(\tilde{x_{i}},x_{i+2})\lesssim[\sigma]_{A_{p'}^{-}}\sigma(F_{k}\cap(\tilde{x_{i}},x_{i+2})).
\end{aligned}
\]
Using this estimate we can continue the computations above as follows
\[
\begin{aligned}(*)\lesssim & 2^{p+1}[w]_{A_{p}^{+}}[\sigma]_{A_{p'}^{-}}\sum_{i=0}^{\infty}\left(\frac{1}{\int_{\tilde{x_{i}}}^{b}\sigma}\int_{\tilde{x_{i}}}^{b}f\right)^{p}\sigma(F_{k}\cap(\tilde{x_{i}},x_{i+2}))\\
 & \leq2^{p+1}[w]_{A_{p}^{+}}[\sigma]_{A_{p'}^{-}}\sum_{i=0}^{\infty}\int_{F_{k}\cap(\tilde{x_{i}},x_{i+2})}(M_{\sigma}(f/\sigma))^{p}\sigma.
\end{aligned}
\]
Gathering the estimates above then we have that 
\[
\int_{\mathbb{R}}(M^{+}f(x))^{p}w(x)dx\leq2^{2p+1}[w]_{A_{p}^{+}}[\sigma]_{A_{p'}^{-}}\sum_{j,k}\sum_{i=0}^{\infty}\int_{F_{k}\cap(\tilde{x_{i}},x_{i+2})}(M_{\sigma}(f/\sigma))^{p}\sigma
\]
Hence it remains to deal with the triple sum in the right hand side.
We argue as follows. First we note that by the definition of the sequence
$\{x_{i}\}_{i=0}^{\infty}$ 
\begin{align}
 & \sum_{j,k}\sum_{i=0}^{\infty}\int_{F_{k}\cap(\tilde{x_{i}},x_{i+2})}(M_{\sigma}(f/\sigma))^{p}\sigma\label{eq:Sumas}\\
\leq & \sum_{j,k}\sum_{i=0}^{\infty}\int_{F_{k}\cap(x_{i},x_{i+1})}(M_{\sigma}(f/\sigma))^{p}\sigma+\int_{F_{k}\cap(x_{i},x_{i+1})}(M_{\sigma}(f/\sigma))^{p}\sigma\nonumber \\
\leq & 2\sum_{k}\sum_{j}\int_{I_{j,k}\cap F_{k}}(M_{\sigma}(f/\sigma))^{p}\sigma\nonumber 
\end{align}
Now we observe that by the definition of the sets $I_{j,k}\cap F_{k}$
they are pairwise disjoint and clearly 
\[
\bigcup_{j}I_{j,k}\cap F_{k}\subset\left\{ x\in\mathbb{R}:\lambda^{k}<\,M^{+}f(x)\leq\lambda^{k+2}\right\} .
\]
Consequently 
\begin{align*}
 & \sum_{k}\sum_{j}\int_{I_{j,k}\cap F_{k}}(M_{\sigma}(f/\sigma))^{p}\sigma\\
\leq & \sum_{k}\int_{\left\{ x\in\mathbb{R}:\lambda^{k}<\,M^{+}f(x)\leq\lambda^{k+2}\right\} }(M_{\sigma}(f/\sigma))^{p}\sigma\\
\leq & \sum_{k}\int_{\left\{ x\in\mathbb{R}:\lambda^{k}<\,M^{+}f(x)\leq\lambda^{k+1}\right\} }(M_{\sigma}(f/\sigma))^{p}\sigma+\int_{\left\{ x\in\mathbb{R}:\lambda^{k+1}<\,M^{+}f(x)\leq\lambda^{k+2}\right\} }(M_{\sigma}(f/\sigma))^{p}\sigma\\
\leq & 2\int_{\mathbb{R}}(M_{\sigma}(f/\sigma))^{p}\sigma
\end{align*}
To end our estimate, we recall that 
\[
\left\Vert M_{\sigma}f\right\Vert _{L^{p}(\sigma)}\leq c\|f\|_{L^{p}(\sigma)}
\]
for some constant independent of $\sigma$ (see \cite{Sj,B}). Consequently
\[
\int_{\mathbb{R}}(M_{\sigma}(f/\sigma))^{p}\sigma\leq c^{p}\int_{\mathbb{R}}|f|^{p}\sigma^{1-p}=\int_{\mathbb{R}}|f|^{p}w.
\]
Gathering the estimates above we are done.}

\subsubsection{Proof of Theorem \ref{thm:mixed}}

Recall that that by the properties of the weak $L^{p}$ norm, 
\[
\|w^{\frac{1}{p}}M^{+}f\|_{L^{p,\infty}}=\|w(M^{+}f)^{p}\|_{L^{1,\infty}}^{\frac{1}{p}}.
\]
Then 
\begin{align*}
\|w(M^{+}f)^{p}\|_{L^{1,\infty}} & =\sup_{t>0}t\left|\left\{ x\in\mathbb{R}\,:\,w(M^{+}f)^{p}>t\right\} \right|\\
 & =\sup_{t>0}\sum_{k}t\left|\left\{ x\in O_{k}\setminus O_{k+1}\,:\,w(M^{+}f)^{p}>t\right\} \right|\\
 & \leq\sup_{t>0}\sum_{k}t\left|\left\{ x\in O_{k}\setminus O_{k+1}\,:\,w2^{(k+1)p}>t\right\} \right|\\
 & \leq\sup_{t>0}\sum_{k,j}2^{(k+1)p}\frac{t}{2^{(k+1)p}}\left|\left\{ x\in E_{jk}\,:\,w>\frac{t}{2^{(k+1)p}}\right\} \right|
\end{align*}
Now we focus on $2^{(k+1)p}\frac{t}{2^{(k+1)p}}\left|\left\{ x\in E_{jk}\,:\,w>\frac{t}{2^{(k+1)p}}\right\} \right|$.
Let us call $I_{jk}=(a,b)$. We split this interval as follows 
\[
\int_{x_{i+1}}^{b}\sigma=\int_{x_{i}}^{x_{i+1}}\sigma.
\]
where $x_{0}=a$. Consequently $\int_{x_{i}}^{b}\sigma=\frac{1}{2^{i}}\int_{a}^{b}\sigma.$
Then 
\begin{align*}
 & 2^{(k+1)p}\frac{t}{2^{(k+1)p}}\left|\left\{ x\in E_{jk}\,:\,w>\frac{t}{2^{(k+1)p}}\right\} \right|\\
= & \sum_{i=0}^{\infty}2^{(k+1)p}\frac{t}{2^{(k+1)p}}\left|\left\{ x\in E_{jk}\cap(x_{i},x_{i+1})\,:\,w>\frac{t}{2^{(k+1)p}}\right\} \right|
\end{align*}
If we call $\tilde{x_{i}}=\inf\left\{ z\in E_{j,k}\cap(x_{i},x_{i+1})\right\} $
and we take into account the properties of each $I_{j,k}$, 
\begin{align*}
 & \sum_{i=0}^{\infty}2^{(k+1)p}\frac{t}{2^{(k+1)p}}\left|\left\{ x\in E_{jk}\cap(x_{i},x_{i+1})\,:\,w>\frac{t}{2^{(k+1)p}}\right\} \right|\\
= & \sum_{i=0}^{\infty}2^{(k+1)p}\frac{t}{2^{(k+1)p}}\left|\left\{ x\in E_{jk}\cap(\tilde{x}_{i},x_{i+1})\,:\,w>\frac{t}{2^{(k+1)p}}\right\} \right|\\
\leq & 2^{p}\sum_{i=0}^{\infty}\left(\frac{1}{\sigma(\tilde{x_{i}},b)}\int_{\tilde{x_{i}}}^{b}f\right)^{p}\frac{\sigma(\tilde{x_{i}},b)^{p}}{(b-\tilde{x_{i}})^{p}}\frac{t}{2^{(k+1)p}}\left|\left\{ x\in E_{jk}\cap(\tilde{x}_{i},x_{i+1})\,:\,w>\frac{t}{2^{(k+1)p}}\right\} \right|\\
\leq & 2^{2p-1}\sum_{i=0}^{\infty}\left(\frac{1}{\sigma(\tilde{x_{i}},b)}\int_{\tilde{x_{i}}}^{b}f\right)^{p}\frac{\sigma(x_{i+1},b)^{p-1}}{(b-\tilde{x_{i}})^{p}}\frac{t}{2^{(k+1)p}}\left|\left\{ x\in(\tilde{x}_{i},x_{i+1})\,:\,w>\frac{t}{2^{(k+1)p}}\right\} \right|\sigma(\tilde{x_{i}},b)\\
\leq & 2^{2p-1}\sum_{i=0}^{\infty}\left(\frac{1}{\sigma(\tilde{x_{i}},b)}\int_{\tilde{x_{i}}}^{b}f\right)^{p}\sigma(x_{i+1},b)^{p-1}\frac{1}{(b-\tilde{x_{i}})^{p}}\|w\chi_{(\tilde{x}_{i},x_{i+1})}\|_{L^{1,\infty}}\sigma(\tilde{x_{i}},b)\\
\leq & [w]_{A_{p}^{+,*}}2^{2p-1}\sum_{i=0}^{\infty}\left(\frac{1}{\sigma(\tilde{x_{i}},b)}\int_{\tilde{x_{i}}}^{b}f\right)^{p}\sigma(\tilde{x_{i}},b)=(*)
\end{align*}
At this point note that since $\tilde{x}_{i}\in E_{j,k}$ then $Mf(\tilde{x}_{i})\leq2^{k+1}$
and then, by Lemma \ref{lem:Key} we have that for every $y>\tilde{x}_{i}$
we have that
\[
|(\tilde{x}_{i},y)|\leq\frac{1}{1-\frac{2^{k+1}}{2^{k+2}}}|F_{k}\cap(\tilde{x}_{i},y)|=2|F_{k}\cap(\tilde{x}_{i},y)|
\]
Taking into account the preceding line, Lemma \ref{lem:CondWeightAp*}
yields that\textcolor{black}{
\[
\begin{aligned}\sigma(\tilde{x_{i}},b) & \leq\sigma(x_{i},b)=4\sigma(x_{i+1},x_{i+2})\\
 & \leq4\sigma(\tilde{x_{i}},x_{i+2})\lesssim[w]_{A_{p}^{+,*}}\sigma(F_{k}\cap(\tilde{x_{i}},x_{i+2})).
\end{aligned}
\]
Using this estimate we can continue the computations above as follows
\[
\begin{aligned}(*)\leq & 2^{p+1}[w]_{A_{p}^{+,*}}^{2}\sum_{i=0}^{\infty}\left(\frac{1}{\int_{\tilde{x_{i}}}^{b}\sigma}\int_{\tilde{x_{i}}}^{b}f\right)^{p}\sigma(F_{k}\cap(\tilde{x_{i}},x_{i+2}))\\
 & \leq2^{p+1}[w]_{A_{p}^{+,*}}^{2}\sum_{i=0}^{\infty}\int_{F_{k}\cap(\tilde{x_{i}},x_{i+2})}(M_{\sigma}(f/\sigma))^{p}\sigma
\end{aligned}
\]
From this point arguing as we did in (\ref{eq:Sumas}) we are done.}

\subsection{Proof of Theorem \ref{thm:mixedFrac}. Sufficiency}

By Lemma \ref{lem:LemRed} and the properties of Lorentz spaces, 
\[
\|wM_{\alpha}^{+}f\|_{L^{q,\infty}}\leq\left(\int_{\mathbb{R}}f(y)^{p}w(y)^{p}dy\right)^{\alpha}\|w^{\frac{q}{s}}M^{+}(f^{p/s}w^{p/s-q/s})(x)\|_{L^{s,\infty}}^{\frac{s}{q}}.
\]
Now we recall that, also by Lemma \ref{lem:ApqAs} $w^{q}\in A_{s}^{+,*}$
for $s=1+\frac{q}{p'}$ and $[w]_{A_{p,q}^{+,*}}=[w^{q}]_{A_{s}^{+,*}}^{\frac{1}{q}}$.
Then, invoking Theorem \ref{thm:mixed}, we have that 
\begin{align*}
\|(w^{q})^{\frac{1}{s}}M(f^{p/s}w^{p/s-q/s})\|_{L^{s,\infty}}^{\frac{s}{q}} & \lesssim\left([w^{q}]_{A_{s}^{+,*}}^{\frac{2}{s}}\right)^{\frac{s}{q}}\|f^{p/s}w^{p/s-q/s}\|_{L^{s}(w^{q})}^{\frac{s}{q}}\\
 & =[w]_{A_{p,q}^{+,*}}^{2}\|f^{p/s}w^{p/s-q/s}\|_{L^{s}(w^{q})}^{\frac{s}{q}}.
\end{align*}
Now we note that working on the right hand side, 
\[
\|f^{p/s}w^{p/s-q/s}\|_{L^{s}(w^{q})}^{\frac{s}{q}}=\left(\int_{\mathbb{R}}f^{p}w^{p-q+q}\right)^{\frac{1}{q}}=\left(\int_{\mathbb{R}}f^{p}w^{p}\right)^{\frac{1}{q}}.
\]
Finally, gathering the estimates above, 
\[
\|wM_{\alpha}^{+}f\|_{L^{q,\infty}}\lesssim[w]_{A_{p,q}^{+,*}}^{2}\left(\int_{\mathbb{R}}f(y)^{p}w(y)^{p}dy\right)^{\alpha+\frac{1}{q}}=[w]_{A_{p,q}^{+,*}}^{2}\left(\int_{\mathbb{R}}f(y)^{p}w(y)^{p}dy\right)^{1/p},
\]
as we wanted to show. 
\begin{rem}
It is worth noting that the argument we have just provided, with the
obvious changes, yields a new proof of the sufficiency due to Sweeting
for the classical setting. 
\end{rem}

\subsection{Proof of Theorems \ref{thm:mixed} and \ref{thm:mixedFrac}. Necessity} Before settling the necessity, we present two results that show, in a  quantitative way, that the $A_{p}^{+,*}$ and the
$A_{p,q}^{+,*}$ conditions hold if, instead of taking supremum over
any subdivision of each interval, we just take supremum subdividing
by the middle point.
\begin{prop}
\label{prop:middleAp*}Let $1<p<\infty$. Let us define
\[
[w]_{\widetilde{A_{p}^{+,*}}}=\sup_{a<b<c,\,b=\frac{a+c}{2}}\frac{1}{(c-a)^{p}}\|w\chi_{(a,b)}\|_{L^{1,\infty}}\left(\int_{b}^{c}w^{-\frac{1}{p-1}}\right)^{p-1}.
\]
Then
\[
[w]_{\widetilde{A_{p}^{+,*}}}\leq[w]_{A_{p}^{+,*}}\leq2^{p}[w]_{\widetilde{A_{p}^{+,*}}}.
\]
\end{prop}

\begin{proof}
The first inequality is trivial, the supremum in the definition of
$[w]_{A_{p}^{+,*}}$ is taken over every subdivision of any interval
whilst, in the case of $[w]_{\widetilde{A_{p}^{+,*}}}$, just the subdivision
of the interval in two equal parts is considered in the supremum.
Conversely, let $a<b<c$ and $m=\frac{a+c}{2}.$

If $b=m$ there's nothing to do. If $a<m<b$ then we have that 
\begin{align*}
 & \frac{1}{(c-a)^{p}}\|w\chi_{(a,b)}\|_{L^{1,\infty}}\left(\int_{b}^{c}w^{-\frac{1}{p-1}}\right)^{p-1}\\
\leq & \frac{(b+b-a-a)^{p}}{(c-a)^{p}}\frac{1}{(b+b-a-a)^{p}}\|w\chi_{(a,b)}\|_{L^{1,\infty}}\left(\int_{b}^{b+b-a}w^{-\frac{1}{p-1}}\right)^{p-1}\\
\leq & \frac{2^{p}(c-a)^{p}}{(c-a)^{p}}\frac{1}{(b+b-a-a)^{p}}\|w\chi_{(a,b)}\|_{L^{1,\infty}}\left(\int_{b}^{b+b-a}w^{-\frac{1}{p-1}}\right)^{p-1}\\
= & 2^{p}\frac{1}{(b+b-a-a)^{p}}\|w\chi_{(a,b)}\|_{L^{1,\infty}}\left(\int_{b}^{b+b-a}w^{-\frac{1}{p-1}}\right)^{p-1}
\end{align*}
from which the desired conclusion follows.

Analogously, if $b<m<c$ then, 
\begin{align*}
 & \frac{1}{(c-a)^{p}}\|w\chi_{(a,b)}\|_{L^{1,\infty}}\left(\int_{b}^{c}w^{-\frac{1}{p-1}}\right)^{p-1}\\
\leq & \frac{(c-(b-(c-b)))^{p}}{(c-a)^{p}}\frac{1}{(c-(b-(c-b)))^{p}}\|w\chi_{(b-(c-b),b)}\|_{L^{1,\infty}}\left(\int_{b}^{c}w^{-\frac{1}{p-1}}\right)^{p-1}\\
\leq & \frac{2^{p}(c-a)^{p}}{(c-a)^{p}}\frac{1}{(c-(b-(c-b)))^{p}}\|w\chi_{(b-(c-b),b)}\|_{L^{1,\infty}}\left(\int_{b}^{c}w^{-\frac{1}{p-1}}\right)^{p-1}\\
= & 2^{p}\frac{1}{(c-(b-(c-b)))^{p}}\|w\chi_{(b-(c-b),b)}\|_{L^{1,\infty}}\left(\int_{b}^{c}w^{-\frac{1}{p-1}}\right)^{p-1}
\end{align*}
from which, again, the desired conclusion readily follows
\end{proof}
The proof of our next result is analogous to the one we have just
presented, hence, we omit it.
\begin{prop}
\label{prop:middleApq*}Let $1<p<\infty$. Let us define
\[
[w]_{\widetilde{A_{p,q}^{+,*}}}=\sup_{a<b<c,\,b=\frac{a+c}{2}}\left(\frac{1}{c-a}\|w^{q}\chi_{(a,b)}\|_{L^{1,\infty}}\right)^{\frac{1}{q}}\left(\frac{1}{c-a}\int_{b}^{c}w^{-p'}\right)^{\frac{1}{p'}}.
\]
Then
\[
[w]_{\widetilde{A_{p,q}^{+,*}}}\leq[w]_{A_{p,q}^{+,*}}\leq2^{\frac{1}{p'}+\frac{1}{q}}[w]_{\widetilde{A_{p,q}^{+,*}}}.
\]
\end{prop}

\subsubsection{Necessity in Theorem \ref{thm:mixed}}
Suppose that 
\begin{equation}
\|w^{\frac{1}{p}}M^{+}f\|_{L^{p,\infty}}\leq K\|f\|_{L^{p}(w)},\label{NM+}
\end{equation}
holds for each $f\in L^{p}(w)$, and let us fix $a\in\R$ and $h>0$.

First of all, let us observe that it is not necessary to assume that
$\sigma=w^{1-p'}$ is locally integrable or that $\sigma(x)>0$ for
a.e. $x\in\R$:

On the one hand, suppose that $\sigma((a,a+h))=0$. Then under the convention
$0\cdot\infty=0$, it is clear that 
\[
\dfrac{1}{h}\|w\chi_{(a-h,a)}\|_{L^{1,\infty}}\left(\dfrac{1}{h}\int_{a}^{a+h}\sigma\right)^{p-1}=0.
\]

On the other hand, suppose that $\sigma((a,a+h))=\infty$. This implies
that $w^{-\frac{1}{p}}\notin L^{p'}((a,a+h))$, since $(w^{-\frac{1}{p}})^{p'}=\sigma$.
Therefore, there exists $g\in L^{p}((a,a+h))$ such that $gw^{-\frac{1}{p}}\notin L^{1}((a,a+h))$.
It follows that $M^{+}(gw^{-\frac{1}{p}})(x)=\infty$ for all $x<a$.

Moreover, since $\|gw^{-\frac{1}{p}}\|_{L^{p}(w)}=\|g\|_{L^{p}}<\infty$,
condition (\ref{NM+}) applied to $f=gw^{-\frac{1}{p}}$ implies that
$w(x)=0$ for a.e. $x\in(-\infty,a)$. Therefore, it follows that
\[
\dfrac{1}{h}\|w\chi_{(a-h,a)}\|_{L^{1,\infty}}\left(\dfrac{1}{h}\int_{a}^{a+h}\sigma\right)^{p-1}=0.
\]

Suppose then that $0<\sigma((a,a+h))<\infty$ and take $f=\sigma\chi_{(a,a+h)}$.
On the one hand, 
\begin{equation}
\|f\|_{L^{p}(w)}=\left(\int_{-\infty}^{\infty}\left(\sigma\chi_{(a,a+h)}\right)^{p}w\right)^{\frac{1}{p}}=\left(\int_{a}^{a+h}\left(w^{1-p'}\right)^{p}w\right)^{\frac{1}{p}}=(\sigma(a,a+h))^{\frac{1}{p}}.\label{NM+1}
\end{equation}

On the other hand, since $M^{+}f$ is non-decreasing on $(a-h,a)$,
for each $x\in(a-h,a)$, we have that 
\begin{align*}
M^{+}f(x) & \geq M^{+}f(a-h)=\sup_{y>0}\dfrac{1}{y}\int_{a-h}^{a-h+y}\sigma(t)\chi_{(a,a+h)}(t)dt\\
 & \geq\dfrac{1}{2h}\int_{a-h}^{a+h}\sigma(t)\chi_{(a,a+h)}(t)dt=\dfrac{1}{2h}\sigma((a,a+h)).
\end{align*}

Hence
\begin{align*}
\left\Vert w^{\frac{1}{p}}M^{+}f\right\Vert _{L^{p,\infty}} & \geq\left\Vert \chi_{(a-h,a)}w^{\frac{1}{p}}M^{+}f\right\Vert _{L^{p,\infty}}\geq\left\Vert \chi_{(a-h,a)}\dfrac{1}{2h}\sigma((a,a+h))w^{\frac{1}{p}}\right\Vert _{L^{p,\infty}}\\
 & =\dfrac{1}{2h}\sigma((a,a+h))\left\Vert \chi_{(a-h,a)}w^{\frac{1}{p}}\right\Vert _{L^{p,\infty}}\\
 & =\dfrac{1}{2h}\sigma((a,a+h))\left\Vert \chi_{(a-h,a)}w\right\Vert _{L^{1,\infty}}^{\frac{1}{p}}
\end{align*}

Gathering the estimates above and using \eqref{NM+}
\[
\dfrac{1}{2h}\sigma((a,a+h))\left\Vert \chi_{(a-h,a)}w\right\Vert _{L^{1,\infty}}^{\frac{1}{p}}\leq K(\sigma(a,a+h))^{\frac{1}{p}}.
\]
This yields
\[
\left(\dfrac{1}{2h}\right)^{p}\sigma((a,a+h))^{p-1}\left\Vert \chi_{(a-h,a)}w\right\Vert _{L^{1,\infty}}\leq K^{p}
\]
and by Proposition \ref{prop:middleAp*} we are done.

\subsubsection{Necessity in Theorem \ref{thm:mixedFrac}}
Assume that
\begin{equation}
\|wM_{\alpha}^{+}f\|_{L^{q,\infty}}\leq K\|f\|_{L^{p}(w^{p})},\label{NF+}
\end{equation}

holds for all $f\in L^{p}(w^{p})$ and let us fix $a\in\R$ and $h>0$.

First of all, let us observe that it is not necessary to assume that
$\sigma=w^{-p'}$ is locally integrable or that $\sigma(x)>0$ for almost every $x\in\R$:

On the one hand, if $\sigma((a,a+h))=0$, then under the convention
$0\cdot\infty=0$, it is clear that 
\[
\left(\dfrac{1}{h}\|w^{q}\chi_{(a-h,a)}\|_{L^{1,\infty}}\right)^{\frac{1}{q}}\left(\dfrac{1}{h}\d\int_{a}^{a+h}\sigma\right)^{\frac{1}{p'}}=0.
\]

On the other hand, suppose that $\sigma((a,a+h))=\infty$. This implies
that $w^{-1}\notin L^{p'}((a,a+h))$. Therefore, there exists $g\in L^{p}((a,a+h))$
such that $gw^{-1}\notin L^{1}((a,a+h))$. It follows that $M_{\alpha}^{+}(gw^{-1})(x)=\infty$
for all $x<a$.

Moreover, since $\|gw^{-1}\|_{L^{p}(w^{p})}=\|g\|_{L^{p}}<\infty$,
 \eqref{NF+} applied to $f=gw^{-1}$ implies that $w(x)=0$
for a.e. $x\in(-\infty,a)$. Therefore, it follows that 
\[
\left(\dfrac{1}{h}\|w^{q}\chi_{(a-h,a)}\|_{L^{1,\infty}}\right)^{\frac{1}{q}}\left(\dfrac{1}{h}\d\int_{a}^{a+h}\sigma\right)^{\frac{1}{p'}}=0.
\]

Suppose then that $0<\sigma((a,a+h))<\infty$ and take $f=\sigma\chi_{(a,a+h)}$.
On the one hand, 
\begin{equation}
\|f\|_{L^{p}(w^{p})}=\left(\d\int_{-\infty}^{\infty}\left(\sigma\chi_{(a,a+h)}\right)^{p}w^{p}\right)^{\frac{1}{p}}=\left(\d\int_{a}^{a+h}\left(w^{-p'}\right)^{p}w^{p}\right)^{\frac{1}{p}}=(\sigma(a,a+h))^{\frac{1}{p}}.\label{NF+1}
\end{equation}

Alternatively, since $M_{\alpha}^{+}f$ is non-decreasing on $(a-h,a)$,
for each $x\in(a-h,a)$, we have that 
\begin{align*}
M_{\alpha}^{+}f(x) & \geq M_{\alpha}^{+}f(a-h)=\sup_{y>0}\dfrac{1}{y^{1-\alpha}}\d\int_{a-h}^{a-h+y}\sigma(t)\chi_{(a,a+h)}(t)dt\\
 & \geq\dfrac{1}{(2h)^{1-\alpha}}\d\int_{a-h}^{a+h}\sigma(t)\chi_{(a,a+h)}(t)dt=(2h)^{\alpha-1}\sigma((a,a+h)).
\end{align*}

Taking this into account, we obtain that 
\[
\left\Vert wM_{\alpha}^{+}f\right\Vert _{L^{q,\infty}}\geq\left\Vert \chi_{(a-h,a)}wM_{\alpha}^{+}f\right\Vert _{L^{q,\infty}}\geq(2h)^{\alpha-1}\sigma((a,a+h))\left\Vert \chi_{(a-h,a)}w\right\Vert _{L^{q,\infty}}
\]
Combining this with (\ref{NF+1}) and \ref{NF+} we have that 
\[
(2h)^{\alpha-1}\sigma((a,a+h))\left\Vert \chi_{(a-h,a)}w\right\Vert _{L^{q,\infty}}\leq K(\sigma(a,a+h))^{\frac{1}{p}}
\]
Finally, taking into account that $\left\Vert w\chi_{(a-h,a)}\right\Vert _{L^{q,\infty}}=\left\Vert w^{q}\chi_{(a-h,a)}\right\Vert _{L^{1,\infty}}^{\frac{1}{q}}$
and also that $\alpha-1=-\frac{1}{p'}-\frac{1}{q}$, the estimate above yields
\[
\left(\dfrac{1}{2h}\left\Vert w^{q}\chi_{(a-h,a)}\right\Vert _{L^{1,\infty}}\right)^{\frac{1}{q}}\left(\dfrac{1}{2h}\d\int_{a}^{a+h}\sigma\right)^{\frac{1}{p'}}\leq K.
\]
At this point, Proposition \ref{prop:middleApq*} ends the proof.
\section{A quantitative two weight estimate }\label{sec:2W}

Before presenting our theorem, we need a few definitions and results
on Young functions and Luxemburg norms. The interested reader can
gain further insight into those topics in \cite{KR}, \cite{RR}.

A function $\Phi:[0,\infty)\rightarrow[0,\infty)$ is said to be a
Young function if $\Phi$ is continuous, convex, and also $\Phi(0)=0$. Since $\Phi$ is convex,  $\frac{\Phi(t)}{t}$
is not decreasing as well.

The Luxemburg average of a function $f$ in terms of a Young function
$\Phi$ on an interval $I$ is defined by 
\[
\|f\|_{\Phi,I}:=\inf\left\{ \lambda>0:\,\frac{1}{|I|}\int_{I}\Phi\left(\frac{|f|}{\lambda}\right)d\mu\leq1\right\} 
\]
We would like to note that if $\Phi(t)=t^{r}$, $r\geq1$, then $\|f\|_{\Phi,I}=\left(\frac{1}{|I|}\int_{I}|f|^{r}\right)^{1/r}$,
namely, we recover the $L^{r}\left(I,\frac{dx}{|I|}\right)$ norm.
Below, we list some interesting properties. \begin{enumerate} 
\item If $I\subset I'$ are intervals
such that $\rho\geq\frac{|I'|}{|I|}>1$, then 
\begin{equation}
\|f\|_{\Phi,I}\leq\rho\|f\|_{\Phi,I'}.\label{eq:Agrandar}
\end{equation}
\item  If $\Phi,\Psi$ are Young
functions such that $\Phi(t)\leq\kappa\Psi(t)$ for all $t\geq c$,
then 
\[
\|f\|_{\Phi,I}\leq(\Phi(c)+\kappa)\|f\|_{\Psi,I}
\]
for every interval $I$. 
\item If $\Phi$
and $\overline{\Phi}$ are Young functions such that for every $t>0$
\[
\Phi^{-1}(t)\overline{\Phi}^{-1}(t)\leq\kappa t,
\]
then for every interval $I$, we have that 
\begin{equation}
\frac{1}{|I|}\int_{I}|fg|d\mu\leq2\kappa\|f\|_{\Phi,I}\|g\|_{\overline{\Phi},I}.\label{eq:HIneqGen}
\end{equation}
\end{enumerate}
Armed with the definitions and properties we have just presented, we define one sided
counterparts of the generalizations of the $A_{\infty}$ and $A_{p}$ constants introduced \cite{PR}. 
\begin{defn}
Given $p>1$ a Young function $\Phi$ and weights $w$ and $\sigma$,
\[
[\sigma,\Phi]_{W_{p}^{-}}=\sup_{I}\frac{1}{\sigma(I)}\int_{I}M_{\Phi}^{+}\left(\sigma^{\frac{1}{p}}\chi_{I}\right)^{p}
\]
\[
[w,\sigma,\Phi]_{A_{p}^{+}}=\sup_{a<b<c}\frac{w(a,b)}{c-a}\|\sigma^{\frac{1}{p'}}\chi_{(b,c)}\|_{\Phi,(a,c)}^{p}
\]
\end{defn}

\begin{rem}
\label{rem:Choices} Note that if $\Phi(t)=t^{p'}$ then we have that
if $\sigma=w^{-\frac{1}{p-1}}$ then 
\[
\|\sigma^{\frac{1}{p'}}\chi_{(b,c)}\|_{\Phi,(a,c)}^{p}=\left(\frac{1}{c-a}\int_{b}^{c}\sigma\right)^{\frac{p}{p'}}=\left(\frac{1}{c-a}\int_{b}^{c}\sigma\right)^{p-1}.
\]
Hence, if additionally $\sigma=w^{-\frac{1}{p-1}}$, then 
\[
[w,\sigma,\Phi]_{A_{p}^{+}}=[w]_{A_{p}^{+}}.
\]
Also, if $\Phi(t)=t^{p}$, 
\[
M_{\Phi}^{+}\left(\sigma^{\frac{1}{p}}\chi_{I}\right)^{p}=M_{p}^{+}\left(\sigma^{\frac{1}{p}}\chi_{I}\right)^{p}=M^{+}(\sigma\chi_{I}).
\]
From this, it readily follows that $[\sigma,\Phi]_{W_{p}^{-}}=[w]_{A_{\infty}^{-}}$. 
\end{rem}

\begin{thm}
\label{thm:2W} Let $\Phi$ and $\overline{\Phi}$ be Young functions
such that for every $t>0,$ 
\[
\Phi^{-1}(t)\overline{\Phi}^{-1}(t)\leq\kappa t.
\]
Then, for every $p>1$ 
\[
\|M^{+}(f\sigma)\|_{L^{p}(w)}\lesssim\left([\sigma,\overline{\Phi}]_{W_{p}^{-}}[w,\sigma,\Phi]_{A_{p}^{+}}\right)^{\frac{1}{p}}\|f\|_{L^{p}(\sigma)}.
\]
\end{thm}

Before proceeding with the proof of this result, we note that, as we
mentioned after Theorem \ref{thm:Fuerte}, if $w\in A_{p}^{+}$ and
$\sigma=w^{-\frac{1}{p-1}}$, then, choosing $\Phi(t)=t^{p'}$ and
$\overline{\Phi}(t)=t^{p}$, we have, by Remark \ref{rem:Choices}, that
\[
\left([\sigma,\overline{\Phi}]_{W_{p}^{-}}[w,\sigma,\Phi]_{A_{p}^{+}}\right)^{\frac{1}{p}}=\left([\sigma]_{A_{\infty}^{-}}[w]_{A_{p}^{+}}\right)^{\frac{1}{p}}
\]
and then \eqref{eq:mixedconstant} follows from the Theorem we have
just stated. 
\begin{proof}[Proof of Theorem \ref{thm:2W}]
It follows by inspection of \cite{MROdT,STesting} that 
\[
\|M^{+}(f\sigma)\|_{L^{p}(w)}\lesssim\sup_{I}\left(\frac{1}{\sigma(I)}\int_{I}M^{+}(\sigma\chi_{I})^{p}wd\mu\right)^{\frac{1}{p}}\|f\|_{L^{p}(\sigma)}.
\]
Consequently it suffices to show that 
\[
\sup_{I}\frac{1}{\sigma(I)}\int_{I}M^{+}(\sigma\chi_{I})^{p}wdx\lesssim[\sigma,\overline{\Phi}]_{W_{p}^{-}}[w,\sigma,\Phi]_{A_{p}^{+}}.
\]
Let us call $I=(\alpha,\beta)$. Let $\lambda>1$. For each $k\in\mathbb{Z}$
we consider 
\[
\Omega_{k}=\left\{ x\in\mathbb{R}\,:\,M^{+}(\sigma\chi_{I})(x)>\lambda^{k}\right\} .
\]
Since for every $k\in\mathbb{Z}$, $\Omega_{k}$ is an open set, there
exists a sequence $\{I_{jk}\}_{j}$ of pairwise disjoint intervals
such that $O_{k}=\bigcup_{j}I_{jk}$ and such that 
\[
\frac{1}{b_{jk}-x}\int_{x}^{b_{jk}}\sigma\chi_{I}>\lambda^{k}\qquad\text{for every }x\in I_{jk}=(a_{jk},b_{jk}).
\]
The intervals $I_{jk}$ are contained in $(-\infty,\beta]$
since if $x>\beta$ then $M^{+}(\sigma\chi_{I})(x)=0$. If additionally
$M^{+}(\sigma\chi_{I})(\alpha)<\lambda^{k}$ then they are actually
contained in $I$. Now we consider 
\[
E_{j,k}=\left\{ x\in I_{jk}:\,M^{+}(\sigma\chi_{I})(x)\leq\lambda^{k+1}\right\} 
\]
It is clear that the sets $E_{j,k}$ are pairwise disjoint. 
\[
\int_{I}M^{+}(\sigma\chi_{I})^{p}w=\int_{\mathbb{R}}M^{+}(\sigma\chi_{I})^{p}w\chi_{I}\leq\lambda^{p}\sum_{k,j}\lambda^{kp}(w\chi_{I})(E_{j,k})=\lambda^{p}\sum_{k,j}\lambda^{kp}w(I\cap E_{j,k})
\]
Now we focus on $\lambda^{kp}w(E_{j,k}\cap I)$ Let us call $(a,b)=\widetilde{I_{jk}}=I_{jk}\cap I=(a_{jk},b_{jk})\cap I$.
We split this interval as follows. We consider a sequence of points
such that $b-x_{i+1}=x_{i+1}-x_{i}$ where $x_{0}=a$. Note that for
at most a finite family of intervals $E_{j,k}\cap(x_{i},x_{i+1})=\emptyset$
and hence $w(E_{j,k}\cap(x_{i},x_{i+1}))=0$. Then if we call $\tilde{x_{i}}=\inf\left\{ z\in E_{j,k}\cap(x_{i},x_{i+1})\right\} $,
we have that 
\begin{align*}
 & \lambda^{kp}w(E_{j,k}\cap I)\\
 & =\lambda^{kp}\sum_{i=0}^{\infty}w(E_{j,k}\cap(x_{i},x_{i+1}))\\
 & =\lambda^{kp}\sum_{i=0}^{\infty}w(E_{j,k}\cap(\tilde{x_{i}},x_{i+1}))\\
{\scriptscriptstyle \left[x_{i+2}\in(a,b)\right]} & \leq\sum_{i=0}^{\infty}\left(\frac{1}{b-x_{i+2}}\int_{x_{i+2}}^{b}\sigma\chi_{I}\right)^{p}w(E_{j,k}\cap(\tilde{x_{i}},x_{i+1}))\\
 & =\sum_{i=0}^{\infty}\left(\frac{1}{b-x_{i+2}}\int_{x_{i+2}}^{b}\sigma^{\frac{1}{p}+\frac{1}{p'}}\chi_{I}\right)^{p}\frac{1}{b-\tilde{x_{i}}}w(E_{j,k}\cap(\tilde{x_{i}},x_{i+1}))(b-\tilde{x_{i}})\\
{\scriptscriptstyle [(\ref{eq:HIneqGen})]} & \leq(2\kappa)^{p}\sum_{i=0}^{\infty}\|\sigma^{\frac{1}{p}}\chi_{I}\|_{\overline{\Phi},(x_{i+2},b)}^{p}\|\sigma^{\frac{1}{p'}}\|_{\Phi,(x_{i+2},b)}^{p}\frac{1}{b-\tilde{x_{i}}}w(\tilde{x_{i}},x_{i+1})(b-x_{i})\\
{\scriptscriptstyle [(\ref{eq:Agrandar})]} & \leq(4\kappa)^{p}\sum_{i=0}^{\infty}\|\sigma^{\frac{1}{p}}\chi_{I}\|_{\overline{\Phi},(x_{i+2},b)}^{p}\|\sigma^{\frac{1}{p'}}\|_{\Phi,(x_{i+1},b)}^{p}\frac{1}{b-\tilde{x_{i}}}w(\tilde{x_{i}},x_{i+1})(b-x_{i})\\
 & \leq4(4\kappa)^{p}[w,\sigma,\Phi]_{A_{p}^{+}}\sum_{i=0}^{\infty}\|\sigma^{\frac{1}{p}}\chi_{I}\|_{\overline{\Phi},(x_{i+2},b)}^{p}(x_{i+2}-x_{i+1})
\end{align*}
Here we define again
\[
F_{k}=\left\{ x\in\mathbb{R}:\,M^{+}f(x)\leq\lambda^{k+2}\right\} .
\]
Then, taking into account Lemma \ref{lem:Key} we have that since
$Mf(\tilde{x}_{i})\leq\lambda^{k+1}$, 
\[
|(\tilde{x_{i}},x_{i+2})|\leq\frac{1}{1-\frac{\lambda^{k+1}}{\lambda^{k+2}}}|(\tilde{x_{i}},x_{i+2})\cap F_{k}|=\lambda'|(\tilde{x_{i}},x_{i+2})\cap F_{k}|.
\]
Bearing this in mind,
\begin{align*}
 & \sum_{i=0}^{\infty}\|\sigma^{\frac{1}{p}}\chi_{I}\|_{\overline{\Phi},(x_{i+2},b)}^{p}(x_{i+2}-x_{i+1})\\
 & \leq\sum_{i=0}^{\infty}\|\sigma^{\frac{1}{p}}\chi_{I}\|_{\overline{\Phi},(x_{i+2},b)}^{p}|(\tilde{x_{i}},x_{i+2})|\\
 & \leq\lambda'\sum_{i=0}^{\infty}\|\sigma^{\frac{1}{p}}\chi_{I}\|_{\overline{\Phi},(x_{i+2},b)}^{p}|F_{j,k}\cap(\tilde{x_{i}},x_{i+2})|\\
{\scriptscriptstyle [(\ref{eq:Agrandar})]} & \leq4\lambda'\sum_{i=0}^{\infty}\left(\inf_{z\in(x_{i},x_{i+2})}M_{\overline{\Phi}}^{+}\left(\sigma^{\frac{1}{p}}\chi_{I}\right)^{p}(z)\right)|F_{k}\cap(x_{i},x_{i+2})|\\
 & \leq4\lambda'\sum_{i=0}^{\infty}\int_{F_{k}\cap(x_{i},x_{i+2})}M_{\overline{\Phi}}^{+}\left(\sigma^{\frac{1}{p}}\chi_{I}\right)^{p}\\
 & =8\lambda'\int_{F_{k}\cap\widetilde{I_{jk}}}M_{\overline{\Phi}}^{+}\left(\sigma^{\frac{1}{p}}\chi_{I}\right)^{p}
\end{align*}
Combining the inequalities above, 
\begin{align*}
\int_{I}M^{+}(\sigma\chi_{I})^{p}w & \lesssim[w,\sigma,\Phi]_{A_{p}^{+}}\sum_{k,j}\int_{F_{k}\cap\widetilde{I_{jk}}}M_{\overline{\Phi}}^{+}\left(\sigma^{\frac{1}{p}}\chi_{I}\right)^{p}\\
 & \lesssim[w,\sigma,\Phi]_{A_{p}^{+}}\int_{I}M_{\overline{\Phi}}^{+}\left(\sigma^{\frac{1}{p}}\chi_{I}\right)^{p}\\
 & \lesssim[w,\sigma,\Phi]_{A_{p}^{+}}[\sigma,\overline{\Phi}]_{W_{p}^{-}}\sigma(I).
\end{align*}
and we are done. 
\section*{Acknowledgements}
The first author would like to thank Prof. Pedro Ortega for suggesting the possibility of extending the Christ-Fefferman approach for the maximal function to the one sided setting.

The first and the second authors were partially supported by the Spanish Ministry of Science and Innovation
through the project PID2022-136619NB-I00 funded by MCIN/AEI/10.13039/501100011033/FEDER, UE and by Junta de Andaluc\'ia through the project FQM-354.

\end{proof}
\bibliographystyle{amsplain}
\bibliography{refs}

\def\cprime{$'$}
\providecommand{\bysame}{\leavevmode\hbox to3em{\hrulefill}\thinspace}
\providecommand{\MR}{\relax\ifhmode\unskip\space\fi MR }
\providecommand{\MRhref}[2]{%
  \href{http://www.ams.org/mathscinet-getitem?mr=#1}{#2}
}
\providecommand{\href}[2]{#2}
\begin{thebibliography}{10}

\bibitem{B}
Antonio Bernal, \emph{A note on the one-dimensional maximal function}, Proc.
  Roy. Soc. Edinburgh Sect. A \textbf{111} (1989), no.~3-4, 325--328.
  \MR{1007529}

\bibitem{MR4002540}
Fabio Berra, \emph{Mixed weak estimates of {S}awyer type for generalized
  maximal operators}, Proc. Amer. Math. Soc. \textbf{147} (2019), no.~10,
  4259--4273. \MR{4002540}

\bibitem{MR4421920}
\bysame, \emph{From {$A_1$} to {$A_\infty$}: new mixed inequalities for certain
  maximal operators}, Potential Anal. \textbf{57} (2022), no.~1, 1--27.
  \MR{4421920}

\bibitem{MR4732437}
\bysame, \emph{Corrigendum to ``{F}rom {$A_1$} to {$A_\infty$}: new mixed
  inequalities for certain maximal operators''}, Potential Anal. \textbf{60}
  (2024), no.~4, 1595--1610. \MR{4732437}

\bibitem{MR3990170}
Fabio Berra, Marilina Carena, and Gladis Pradolini, \emph{Mixed weak estimates
  of {S}awyer type for commutators of generalized singular integrals and
  related operators}, Michigan Math. J. \textbf{68} (2019), no.~3, 527--564.
  \MR{3990170}

\bibitem{MR3987919}
\bysame, \emph{Mixed weak estimates of {S}awyer type for fractional integrals
  and some related operators}, J. Math. Anal. Appl. \textbf{479} (2019), no.~2,
  1490--1505. \MR{3987919}

\bibitem{MR4166766}
\bysame, \emph{Improvements on {S}awyer type estimates for generalized maximal
  functions}, Math. Nachr. \textbf{293} (2020), no.~10, 1911--1930.
  \MR{4166766}

\bibitem{MR4533037}
\bysame, \emph{Mixed inequalities of {F}efferman-{S}tein type for singular
  integral operators}, J. Math. Sci. (N.Y.) \textbf{266} (2022), no.~3,
  461--475. \MR{4533037}

\bibitem{MR4694578}
\bysame, \emph{Corrigendum to ``{I}mprovements on {S}awyer-type estimates for
  generalized maximal functions''}, Math. Nachr. \textbf{296} (2023), no.~12,
  5786--5788. \MR{4694578}

\bibitem{MR4614637}
\bysame, \emph{Mixed inequalities for commutators with multilinear symbol},
  Collect. Math. \textbf{74} (2023), no.~3, 605--637. \MR{4614637}

\bibitem{MR4140763}
Marcela Caldarelli and Israel~P. Rivera-R\'ios, \emph{A sparse approach to
  mixed weak type inequalities}, Math. Z. \textbf{296} (2020), no.~1-2,
  787--812. \MR{4140763}

\bibitem{CF}
Michael Christ and Robert Fefferman, \emph{A note on weighted norm inequalities
  for the {H}ardy-{L}ittlewood maximal operator}, Proc. Amer. Math. Soc.
  \textbf{87} (1983), no.~3, 447--448. \MR{684636}

\bibitem{MR2172941}
D.~Cruz-Uribe, J.~M. Martell, and C.~P\'erez, \emph{Weighted weak-type
  inequalities and a conjecture of {S}awyer}, Int. Math. Res. Not. (2005),
  no.~30, 1849--1871. \MR{2172941}

\bibitem{CUIMPRR}
David Cruz-Uribe, Joshua Isralowitz, Kabe Moen, Sandra Pott, and Israel~P.
  Rivera-R\'ios, \emph{Weak endpoint bounds for matrix weights}, Rev. Mat.
  Iberoam. \textbf{37} (2021), no.~4, 1513--1538. \MR{4269407}

\bibitem{CUS}
David Cruz-Uribe and Brandon Sweeting, \emph{Weighted weak-type inequalities
  for maximal operators and singular integrals}, Rev. Mat. Complut. \textbf{38}
  (2025), no.~1, 183--205. \MR{4859193}

\bibitem{H}
Tuomas~P. Hyt\"onen, \emph{The sharp weighted bound for general
  {C}alder\'on-{Z}ygmund operators}, Ann. of Math. (2) \textbf{175} (2012),
  no.~3, 1473--1506. \MR{2912709}

\bibitem{KR}
Mark~A. Krasnosel{\cprime}ski{\u\i} and Ja.~B. Ruticki{\u\i}, \emph{Convex
  functions and {O}rlicz spaces}, Translated from the first Russian edition by
  Leo F. Boron, P. Noordhoff Ltd., Groningen, 1961. \MR{0126722 (23 \#A4016)}

\bibitem{LLORR}
Andrei~K. Lerner, Kangwei Li, Sheldy Ombrosi, and Israel~P. Rivera-R\'ios,
  \emph{On the sharpness of some quantitative {M}uckenhoupt-{W}heeden
  inequalities}, C. R. Math. Acad. Sci. Paris \textbf{362} (2024), 1253--1260.
  \MR{4824922}

\bibitem{LLORR2}
Andrei~K. Lerner, Kangwei Li, Sheldy Ombrosi, and Israel~P. Rivera-Ríos,
  \emph{On some improved weighted weak type inequalities}, To appear in Ann.
  Sc. Norm. Super. Pisa Cl. Sci. (2024), 21 pages.

\bibitem{MR3961329}
Kangwei Li, Sheldy Ombrosi, and Carlos P\'erez, \emph{Proof of an extension of
  {E}. {S}awyer's conjecture about weighted mixed weak-type estimates}, Math.
  Ann. \textbf{374} (2019), no.~1-2, 907--929. \MR{3961329}

\bibitem{MROdT}
F.~J. Mart\'in-Reyes, P.~Ortega~Salvador, and A.~de~la Torre, \emph{Weighted
  inequalities for one-sided maximal functions}, Trans. Amer. Math. Soc.
  \textbf{319} (1990), no.~2, 517--534. \MR{986694}

\bibitem{MRdT}
Francisco~J. Mart\'in-Reyes and Alberto de~la Torre, \emph{Sharp weighted
  bounds for one-sided maximal operators}, Collect. Math. \textbf{66} (2015),
  no.~2, 161--174. \MR{3338703}

\bibitem{M}
Benjamin Muckenhoupt, \emph{Weighted norm inequalities for the {H}ardy maximal
  function}, Trans. Amer. Math. Soc. \textbf{165} (1972), 207--226. \MR{293384}

\bibitem{MW}
Benjamin Muckenhoupt and Richard~L. Wheeden, \emph{Some weighted weak-type
  inequalities for the {H}ardy-{L}ittlewood maximal function and the {H}ilbert
  transform}, Indiana Univ. Math. J. \textbf{26} (1977), no.~5, 801--816.
  \MR{447956}

\bibitem{NArxiv}
Zoe Nieraeth, \emph{A lattice approach to matrix weights}, Math. Ann.
  \textbf{393} (2025), no.~1, 993--1072. \MR{4966577}

\bibitem{MR3557137}
Sheldy Ombrosi and Carlos P\'erez, \emph{Mixed weak type estimates: examples
  and counterexamples related to a problem of {E}. {S}awyer}, Colloq. Math.
  \textbf{145} (2016), no.~2, 259--272. \MR{3557137}

\bibitem{MR3498179}
Sheldy Ombrosi, Carlos P\'erez, and Jorgelina Recchi, \emph{Quantitative
  weighted mixed weak-type inequalities for classical operators}, Indiana Univ.
  Math. J. \textbf{65} (2016), no.~2, 615--640. \MR{3498179}

\bibitem{O}
P.~Ortega~Salvador, \emph{Weighted {L}orentz norm inequalities for the
  one-sided {H}ardy-{L}ittlewood maximal functions and for the maximal ergodic
  operator}, Canad. J. Math. \textbf{46} (1994), no.~5, 1057--1072.
  \MR{1295131}

\bibitem{PR}
Carlos P\'erez and Ezequiel Rela, \emph{A new quantitative two weight theorem
  for the {H}ardy-{L}ittlewood maximal operator}, Proc. Amer. Math. Soc.
  \textbf{143} (2015), no.~2, 641--655. \MR{3283651}

\bibitem{RR}
Malempati~M. Rao and Zhong~D. Ren, \emph{Theory of {O}rlicz spaces}, Monographs
  and Textbooks in Pure and Applied Mathematics, vol. 146, Marcel Dekker, Inc.,
  New York, 1991. \MR{1113700}

\bibitem{RdRdT}
M.~S. Riveros, L.~de~Rosa, and A.~de~la Torre, \emph{Sufficient conditions for
  one-sided operators}, J. Fourier Anal. Appl. \textbf{6} (2000), no.~6,
  607--621. \MR{1790246}

\bibitem{SMixed}
E.~Sawyer, \emph{A weighted weak type inequality for the maximal function},
  Proc. Amer. Math. Soc. \textbf{93} (1985), no.~4, 610--614. \MR{776188}

\bibitem{S}
\bysame, \emph{Weighted inequalities for the one-sided {H}ardy-{L}ittlewood
  maximal functions}, Trans. Amer. Math. Soc. \textbf{297} (1986), no.~1,
  53--61. \MR{849466}

\bibitem{STesting}
\bysame, \emph{Weighted inequalities for the one-sided {H}ardy-{L}ittlewood
  maximal functions}, Trans. Amer. Math. Soc. \textbf{297} (1986), no.~1,
  53--61. \MR{849466}

\bibitem{Sj}
Peter Sj\"ogren, \emph{A remark on the maximal function for measures in
  {$\mathbb{R}^n$}}, Amer. J. Math. \textbf{105} (1983), no.~5, 1231--1233.
  \MR{714775}

\bibitem{Sw}
Brandon Sweeting, \emph{On those weights satisfying a weak-type inequality for
  the maximal operator and fractional maximal operator}, Preprint available in
  arXiv (2024), 12 pages.

\end{thebibliography}

\end{document}